\documentclass[11pt]{article}
\usepackage[a4paper,margin=1in]{geometry}
\usepackage{amsmath,amssymb,amsfonts,amsthm,mathtools}
\usepackage{booktabs}
\usepackage{microtype}
\usepackage[hidelinks]{hyperref}
\usepackage{cite}
\newtheorem{theorem}{Theorem}[section]
\newtheorem{lemma}[theorem]{Lemma}
\newtheorem{corollary}[theorem]{Corollary}
\newtheorem{example}[theorem]{Example}

\theoremstyle{definition}

\theoremstyle{remark}
\newtheorem{remark}[theorem]{Remark}

\usepackage{graphicx}
\title{\textbf{Invariance of Hurwitz-Stability of Polynomials of Degree Five under Positive Hadamard Powers}
}
\author{
Rola Alseidi\thanks{Department of Mathematics, Philadelphia University, Amman, Jordan.
Email: \texttt{ralseidi@philadelphia.edu.jo}}
\and
Raoufah Alsaidi\thanks{Department of Mathematics, Jordan University of Science and Technology, Irbid, Jordan.
Email: \texttt{raofe@just.edu.jo}}
\and
J\"urgen Garloff\thanks{Department of Mathematics and Statistics, University of Konstanz, 78464 Konstanz, Germany; HTWG Konstanz -- University of Applied Sciences, Faculty of Computer Science, 78405 Konstanz, Germany.
Email: \texttt{garloff@htwg-konstanz.de}}
}
\date{}
\begin{document}
\maketitle
% Abstract, keywords, AMS classification, and final title are postponed.
 \begin{abstract}
A complete characterization of the stability-preserving exponent set
for fractional Hadamard powers of a monic Hurwitz-stable polynomial  \(f\)
of degree five is presented. The stability problem is reduced to the
analysis of a single scalar function depending only on three parameters
formed from the coefficients of the polynomial \(f\). This reduction
leads to a unique stability threshold \(p_*(f)\in(0,1)\) such that
the \(p\)-th Hadamard power of \(f\) is Hurwitz stable if and only if
\(p>p_*(f)\). Consequently, the stability-preserving exponent set is
precisely \((p_*(f),\infty)\). This threshold depends smoothly on the
parameters and provides a global coordinate on the admissible parameter
region. Finally, smooth dependence is illustrated by a one-parameter
family of Hurwitz-stable polynomials whose complex-conjugate zeros
approach the imaginary axis.
\end{abstract}
\section{Introduction and Preliminaries}
\label{sec:introduction}

A real polynomial is called (Hurwitz-) stable if all its zeros lie in the
open left half of the complex plane. Throughout this paper, we consider
monic stable polynomials having only positive coefficients. Their stability
is characterized by the Routh--Hurwitz criterion or, equivalently, by the
Li\'enard--Chipart criterion, which requires the positivity of an appropriate
subset of the Hurwitz determinants
\cite{Hurwitz1895,LienardChipart1914,Gantmacher1959}.

Let
\[
f(x)=\sum_{k=0}^{n} a_k x^{n-k},
\qquad
g(x)=\sum_{k=0}^{n} b_k x^{n-k}
\]
be two real polynomials. Their \textit{Hadamard product} is defined by
coefficientwise multiplication as
\[
(f\circ g)(x)
=
\sum_{k=0}^{n} a_k b_k x^{n-k}.
\]
If all coefficients of \(f\) are positive, then, for \(p>0\), its
\(p\)-th \textit{Hadamard power} is defined by
\[
f^{[p]}(x)
=
\sum_{k=0}^{n} a_k^{\,p}x^{n-k}.
\]

Garloff and Wagner proved that the Hadamard product of two stable
polynomials is again stable \cite{GarloffWagner1996}. Consequently, \(f^{[p]}\) is stable for every positive integer \(p\)
whenever \(f\) is stable. This preservation property, however, does
not extend to fractional exponents.

The first negative results for fractional Hadamard powers were
obtained by Bia{\l}as and Bia{\l}as-Cie{\.z}, who showed that stability
is not preserved for all real exponents \(p>1\) and established
sufficient conditions guaranteeing stability for sufficiently large
exponents \cite{BialasBialasCiez2017}. Later, Bia{\l}as et al. proved
that every stable polynomial of degree at most five remains stable for
every exponent \(p>1\), and constructed a degree-six polynomial whose
stability-preserving exponent set is not an interval
\cite{BialasBialasCiezKudra2024}.

Related questions have also been studied for Hadamard roots.
Bia{\l}as and G{\'o}ra investigated the relation between Hadamard
factorizability and the stability of Hadamard square roots of
polynomials of degree four \cite{BialasGora2021}. More recently,
Alsaafin et al. obtained a complete characterization of the Hadamard
square root for stable polynomials of degree five
\cite{AlsaafinAlSaafinGarloff2024}.

For degree five, so it is known that \(f^{[p]}\) is stable for every
exponent \(p>1\). But the behavior of fractional Hadamard powers
in the interval \(0<p<1\)  remained unknown. The present paper aims
at filling this gap.

The analysis developed in this paper is based on a parameterization
of the degree-five Hurwitz region, by which the stability problem
reduces to the study of a single scalar function depending only on
three parameters formed from the polynomial coefficients. This
reduction yields a complete characterization of the
stability-preserving exponent set for every stable polynomial of
degree five.

Throughout this paper, let \(H_5\) denote the family of all monic
stable polynomials of degree five with positive coefficients. For
\(f\in H_5\), define
\[
\mathcal P(f)
=
\{\,p>0:\,f^{[p]}\in H_5\,\}.
\]

Our main result is a complete characterization of the
stability-preserving exponent set \(\mathcal P(f)\) for every
polynomial \(f\in H_5\).

The characterization is complemented by several structural properties
of the stability threshold. In particular, we prove that the threshold
depends smoothly on the parameters and provides a global coordinate on
the admissible parameter region. We illustrate its behavior by a
one-parameter family in \(H_5\).

The paper is organized as follows. Section~\ref{sec:main-results}
establishes the complete characterization of the
stability-preserving exponent set. Section~\ref{sec:smooth-threshold}
studies the smooth and geometric structure of the stability threshold
and illustrates its behavior along a one-parameter family of stable
polynomials. Section~\ref{sec:conclusion} summarizes the main results
and discusses directions for future research.
\section{Main Results}
\label{sec:main-results}
This section establishes the complete characterization of the
stability-preserving exponent set for degree-five stable
polynomials. The key step is to reduce the stability conditions
to the analysis of a single scalar function.

\subsection{Reduction to a Scalar Stability Function}
\label{subsec:scalar-reduction}

Let
\[
f(x)
=
x^{5}
+a_{1}x^{4}
+a_{2}x^{3}
+a_{3}x^{2}
+a_{4}x
+a_{5}
\in H_{5},
\]
and, for \(p>0\), let
\[
f^{[p]}(x)
=
x^{5}
+a_{1}^{p}x^{4}
+a_{2}^{p}x^{3}
+a_{3}^{p}x^{2}
+a_{4}^{p}x
+a_{5}^{p}
\]
denote its \(p\)-th Hadamard power.

We assume that \(f\) is stable. By the Li\'enard--Chipart
criterion \cite{LienardChipart1914}, this is equivalent to the positivity
of the second and fourth Hurwitz determinants of \(f\), i.e.,
\[
\Delta_2(f)
=
\det
\begin{pmatrix}
a_1 & a_3\\
1   & a_2
\end{pmatrix}
>0,
\]
and
\[
\Delta_4(f)
=
\det
\begin{pmatrix}
a_1 & a_3 & a_5 & 0\\
1   & a_2 & a_4 & 0\\
0   & a_1 & a_3 & a_5\\
0   & 1   & a_2 & a_4
\end{pmatrix}
>0.
\]

The first inequality is preserved under Hadamard powers. Introduce the parameters
\begin{equation}
u=\frac{a_3}{a_1a_2},
\qquad
v=\frac{a_2a_5}{a_3a_4},
\qquad
w=\frac{a_1a_4}{a_2a_3}.
\label{eq:uvw-definition}
\end{equation}
The positivity of the principal minors of the Hurwitz matrix
\cite[Theorem~2]{Kemperman1982} implies
\begin{equation}
0<u,v,w<1,
\label{eq:uvw-range}
\end{equation}
see also \cite[formula~(1.10)]{Kemperman1982}. In terms of these parameters, the fourth Hurwitz determinant of
\(f^{[p]}\) admits the factorization
\begin{equation}
\Delta_4(f^{[p]})
=
(a_1a_2a_3a_4)^p
\Bigl[
(1-u^p)(1-v^p)
-
w^p\bigl(1-(uv)^p\bigr)^2
\Bigr].
\label{eq:delta4-factorized}
\end{equation}

Since \((a_1a_2a_3a_4)^p>0\), the inequality
\(\Delta_4(f^{[p]})>0\) is equivalent to
\[
(1-u^p)(1-v^p)
>
w^p\bigl(1-(uv)^p\bigr)^2.
\]

Define
\begin{equation}
h_f(p)
=
\ln(1-u^p)
+\ln(1-v^p)
-2\ln\!\bigl(1-(uv)^p\bigr)
-p\ln w.
\label{eq:h}
\end{equation}
By \eqref{eq:uvw-range}, the logarithms are well defined for every
\(p>0\). Therefore, $\Delta_4(f^{[p]})>0$ if and only if $h_f(p)>0,$ 
and the study of stability for the Hadamard powers of \(f\)
reduces to the analysis of the scalar function \(h_f\).
%%%%%%%%%%%%%%%%%%%%%%%%%%%%%%%%%%%%%%%%%%%

\begin{lemma}
\label{lem:h-properties}
Let \(f\in H_5\), and let \(h_f\) be defined by \eqref{eq:h}. Put 
\(u=e^{-\alpha}\) and \(v=e^{-\beta}\), where
\(\alpha,\beta>0\). Then:

\begin{enumerate}
%\item[(i)]  \(h_f\in C^\infty(0,\infty)\).

\item[(i)]  The function \(h_f\) is strictly concave on \((0,\infty)\),

\item[(ii)] \(h_f(1)>0\),

\item[(iii)]
\[
\lim_{p\to0^+}h_f(p)
=
\ln\!\left(\frac{\alpha\beta}{(\alpha+\beta)^2}\right)
<0,
\]

\item[(iv)]
\[
\lim_{p\to\infty}h_f(p)=+\infty.
\]
\end{enumerate}
\end{lemma}
\begin{proof}

%\noindent\textit{(i)}
%This follows directly from the smoothness of the logarithm and
%exponential functions.

%\medskip

\noindent\textit{(i)}
Differentiation with respect to $p$ gives
$
p^{2}h_f''(p)
=
-q(\alpha p)-q(\beta p)+2q((\alpha+\beta)p),
$
where
\(q(x)=x^{2}e^{x}/(e^{x}-1)^{2}\).
Moreover,
\(
q'(x)/q(x)
=
2/x-\coth(x/2).
\)
Let
\(r(x)=x\cosh(x/2)-2\sinh(x/2)\).
Since
\(r(0)=0\)
and
\(r'(x)=\frac{x}{2}\sinh(x/2)>0\)
for \(x>0\),
it follows that
\(r(x)>0\), and hence
\(2/x<\coth(x/2)\).
Therefore
\(q'(x)<0\),
and so \(q\) is strictly decreasing.
Since
\((\alpha+\beta)p>\alpha p\)
and
\((\alpha+\beta)p>\beta p\),
we obtain
$
2q((\alpha+\beta)p)
<
q(\alpha p)+q(\beta p),
$
and consequently
\(h_f''(p)<0\)
for every \(p>0\).

\medskip

\noindent\textit{(ii)}
Since \(f\in H_5\), one has \(\Delta_4(f)>0\). By the factorization
of \(\Delta_4\) obtained above, this is equivalent to
\[
(1-u)(1-v)-w(1-uv)^2>0,
\]
and hence \(h_f(1)>0\).

\medskip

\noindent\textit{(iii)}
By l'Hospital's rule, we obtain
\[
\begin{aligned}
\lim_{p\to0^+}
\ln\left(
\frac{1-u^p}{1-(uv)^p}
\frac{1-v^p}{1-(uv)^p}
\right)
&=
\lim_{p\to0^+}
\ln\left(
\frac{\ln u\,\ln v}
{u^p v^p \ln^2(uv)}
\right)\\
&=
\ln\left(
\frac{\ln u\,\ln v}
{\ln^2(uv)}
\right)
=
\ln\left(
\frac{\alpha\beta}{(\alpha+\beta)^2}
\right)
<0.
\end{aligned}
\]
\medskip

\noindent\textit{(iv)}
As \(p\to\infty\), the first three  logarithmic terms in
\(h_f(p)\) tend to \(0\), whereas
\(-p\ln w\to+\infty\) because \(0<w<1\).
Therefore
\(\lim_{p\to\infty}h_f(p)=+\infty\).

\end{proof}

%%%%%%%%%%%%%%%%%%%%%%%%%%%%%%%%%%%%%%%%%%%%%%%%%%%%%%%%

\begin{lemma}\label{lem:unique-threshold}
For every \(f\in H_5\), the equation \(h_f(p)=0\) has a unique
solution \(p_*(f)\in(0,1)\).
\end{lemma}

\begin{proof}
Existence follows from Lemma~2.1(ii) and (iii) and the continuity of \(h_f\).

Suppose that \(h_f\) has two distinct positive zeros
\(0<p_1<p_2\). By the Mean Value Theorem, there exists
\(\xi\in(p_1,p_2)\) such that \(h_f'(\xi)=0\).
Since \(h_f\) is strictly concave, \(h_f''(\xi)<0\), and therefore
\(\xi\) is a local maximum which is also a global maximum.
But this contradicts Lemma~2.1(iv).

\end{proof}
%%%%%%%%%%%%%%%%%%%%%%%%%%%%%%%%%%%%%%%%%%%%%%%%%%%

\subsection{The Stability Threshold}
The properties of the stability function $h_f$ derived in Lemmata 2.1 and 2.2 are summarized in the next theorem.
\label{subsec:exponent-set}

\begin{theorem}
\label{thm:main}

Let \(f\in H_5\). Then
\[
\mathcal P(f)
=
\{\,p>0:f^{[p]}\in H_5\,\}
=
(p_*(f),\infty),
\]
where \(p_*(f)\in(0,1)\) is the unique zero of \(h_f\). Moreover,
\(p_*(f)\) is the unique solution in \((0,1)\) of the equation
\begin{equation}
(1-u^p)(1-v^p)
=
w^p\bigl(1-(uv)^p\bigr)^2.
\label{eq:threshold-equation}
\end{equation}

\end{theorem}

%%%%%%%%%%%%%%%%%%%%%%%%%%%%%%%%%%%%%55
\begin{corollary}
\label{cor:sqrt}

For \(f\in H_5\), the following are equivalent.

\begin{enumerate}
\item[(i)]
\(f^{[1/2]}\in H_5\).

\item[(ii)]
\(p_*(f)<\frac12.\)

\item[(iii)]
\[
(1-\sqrt u)(1-\sqrt v)
>
\sqrt w\,
(1-\sqrt{uv})^2.
\]

\end{enumerate}

\end{corollary}

\begin{proof}

By Theorem~\ref{thm:main},
\(f^{[1/2]}\in H_5\) if and only if
\(p_*(f)<\frac12\). On the other hand,

\[
h_f\!\left(\frac12\right)
=
\ln\!\left(
\frac{(1-\sqrt u)(1-\sqrt v)}
{\sqrt w\,(1-\sqrt{uv})^2}
\right),
\]

so \(h_f(1/2)>0\) is equivalent to

\[
(1-\sqrt u)(1-\sqrt v)
>
\sqrt w\,
(1-\sqrt{uv})^2.
\]

\end{proof}

\begin{remark}
The threshold \(p_*(f)\) can be computed by using the strict concavity
of \(h_f\). We first observe that \(h_f'(p)>0\) for every \(p>0\).
Indeed, Lemma~2.1(i) shows that \(h_f'\) is strictly decreasing, while
\[
\lim_{p\to\infty} h_f'(p)=-\ln w>0.
\]
Thus
\[
h_f'(p)>-\ln w>0,\qquad p>0.
\]

Starting with \(p_0=1\), define
\[
q=p_0-\frac{h_f(p_0)}{h_f'(p_0)}.
\]
Since
\[
\frac{d}{dp}\bigl(ph_f'(p)-h_f(p)\bigr)
=ph_f''(p)<0
\]
and
\[
\lim_{p\to\infty}\bigl(ph_f'(p)-h_f(p)\bigr)=0,
\]
 one has \(q>0\) (note that the linear terms \(p\ln w\) cancel).
Let \(p_1\) be the zero of the secant line through
\((q,h_f(q))\) and \((p_0,h_f(p_0))\), that is,
\[
p_1
=
q-\frac{h_f(q)(p_0-q)}
{h_f(p_0)-h_f(q)}.
\]
Strict concavity and the sign properties of \(h_f\) imply
\[
0<q<p_*(f)<p_1<p_0.
\]
This step can be iterated to generate a monotonically decreasing
sequence \((p_n)_{n\in\mathbb N}\) converging to \(p_*(f)\), providing
at each step an upper bound for it.
\end{remark}
We conclude this section with two examples illustrating the threshold
characterization.

\begin{example} 
For \(\varepsilon>0\), consider the family 
\(f_\varepsilon(x)=(x+\varepsilon)^5\). Its parameters are 
\(u=v=\frac15\) and \(w=\frac14\), independently of \(\varepsilon\). 
The corresponding stability threshold is 
\(p_*(f_\varepsilon)\approx0.5195\), and therefore 
\(\mathcal P(f_\varepsilon)=(0.5195\ldots,\infty)\). Thus, the threshold is 
independent of \(\varepsilon\), even though all zeros converge to the 
origin as \(\varepsilon\to0^+\). 
\end{example} 
 
\begin{example} 
Consider the polynomial from Example~3.9 of 
\cite{AlsaafinAlSaafinGarloff2024}, 
\[ 
f(x)=x^5+15x^4+10x^3+30x^2+10x+10. 
\] 
Its zeros are approximately 
\[ 
-14.4485,\qquad 
-0.0890\pm1.2238i,\qquad 
-0.1868\pm0.6518i. 
\] 
Its parameters are \(u=\frac15\), \(v=\frac13\), and 
\(w=\frac12\). The corresponding stability threshold is 
\(p_*(f)\approx0.8449\), and therefore 
\(\mathcal P(f)=(0.8449\ldots,\infty)\). Since 
\(p_*(f)>\frac12\), the Hadamard square root is not stable. 
Hence, Corollary~2.4 recovers the conclusion of 
Example~3.9 in \cite{AlsaafinAlSaafinGarloff2024}. 
\end{example}
\section{Smooth and Geometric Structure of the Stability Threshold}
\label{sec:smooth-threshold}

Theorem~\ref{thm:main} shows that every polynomial
\(f\in H_5\) possesses a unique stability threshold
\(p_*(f)\in(0,1)\), which depends only on the parameters
\(u,v,w\). It is therefore natural to regard the threshold as a
function on the  parameter space. For
\((u,v,w)\in(0,1)^3\) and \(p>0\), define
\[h(u,v,w,p)=\ln(1-u^p)+\ln(1-v^p)-2\ln\!\bigl(1-(uv)^p\bigr)-p\ln w,\]
and let
\[
\Omega
=
\Bigl\{
(u,v,w)\in(0,1)^3:
h(u,v,w,1)>0
\Bigr\}.
\]
Then \(h_f(p)=h(u,v,w,p)\) for the  triple associated with
\(f\). Since \(h(u,v,w,1)\) is continuous as a function of
\((u,v,w)\), the set \(\Omega\) is an open subset of
\((0,1)^3\). The arguments of
Lemmata~\ref{lem:h-properties} and~\ref{lem:unique-threshold} therefore
apply to every \((u,v,w)\in\Omega\). Hence the equation
\(h(u,v,w,p)=0\) has a unique solution
\(p_*(u,v,w)\in(0,1)\), satisfying
\(h(u,v,w,p)<0\) for
\(0<p<p_*(u,v,w)\) and
\(h(u,v,w,p)>0\) for
\(p>p_*(u,v,w)\). The following results show that the stability
threshold depends smoothly on the parameters and provides a
global smooth coordinate on the admissible parameter region.

\begin{theorem}
\label{thm:smooth-threshold}

The map
\[
p_*:\Omega\longrightarrow(0,1),\qquad
(u,v,w)\longmapsto p_*(u,v,w),
\]
is of class \(C^\infty\).

\end{theorem}
\begin{proof}
The function \(h:\Omega\times(0,\infty)\to\mathbb{R}\) is of class
\(C^\infty\). Fix \((u,v,w)\in\Omega\), and let
\(p_*=p_*(u,v,w)\). Since \(p\mapsto h(u,v,w,p)\) is strictly
concave, \(h(u,v,w,p_*)=0\), and \(h(u,v,w,1)>0\), it follows that
\[
h(u,v,w,1)
<
h(u,v,w,p_*)
+
\partial_p h(u,v,w,p_*)(1-p_*),
\]
whence \(\partial_p h(u,v,w,p_*)>0\).

The Implicit Function Theorem therefore determines \(p\) locally as a
\(C^\infty\)-function of \((u,v,w)\). Since the defining zero is unique
by the preceding discussion, these local representations agree wherever
their domains intersect and determine a global \(C^\infty\)-map
\[
p_*:\Omega\longrightarrow(0,1).
\]
\end{proof}
%%%%%%%%%%%%%%%%%%%%%%%%%%%%%%%%%%%%%%%%%%%%%%%%%%%%%%%%%%%%%%%%%%%%%%%%
Theorem~\ref{thm:smooth-threshold} shows that the stability threshold
depends smoothly on the  parameters. Since the defining
equation can be solved explicitly for the third  parameter,
the threshold itself may be used in place of \(w\). The next result
shows that this yields a global smooth parametrization of the
admissible parameter region.

\begin{theorem}
\label{thm:global-coordinates}

Let
\[
\Phi:(0,1)^2\times(0,1)\longrightarrow\Omega
\]
be defined by
\[
\Phi(u,v,p)
=
\left(
u,
v,
\left(
\frac{(1-u^p)(1-v^p)}
{\left(1-(uv)^p\right)^2}
\right)^{1/p}
\right).
\]
Then \(\Phi\) is a \(C^\infty\)-diffeomorphism between
\((0,1)^2\times(0,1)\) and \(\Omega\), with inverse
\[
\Phi^{-1}(u,v,w)
=
\bigl(u,v,p_*(u,v,w)\bigr).
\]
In particular, the stability threshold provides a global smooth
coordinate on the admissible parameter region \(\Omega\).

\end{theorem}

\begin{proof}
Let \((u,v,p)\in(0,1)^2\times(0,1)\), and let \(w\) be the third
component of \(\Phi(u,v,p)\). Since \(0<u,v<1\), we have
\[
1-(uv)^p>1-u^p
\qquad\text{and}\qquad
1-(uv)^p>1-v^p.
\]
Hence
\[
(1-u^p)(1-v^p)
<
\bigl(1-(uv)^p\bigr)^2.
\]

and therefore \(0<w<1\). By the definition of \(w\),
\[
h(u,v,w,p)=0.
\]
For fixed \(u,v,w\in(0,1)\), the function
\(t\mapsto h(u,v,w,t)\) is strictly concave. Hence its derivative is
strictly decreasing, and
\[
\lim_{t\to\infty}\partial_t h(u,v,w,t)=-\ln w>0.
\]
Thus \(\partial_t h(u,v,w,t)>0\) for \(t>0\), so
\(t\mapsto h(u,v,w,t)\) is strictly increasing. Since \(p<1\),
\[
h(u,v,w,1)>h(u,v,w,p)=0.
\]
Hence \((u,v,w)\in\Omega\). The uniqueness of the zero gives
\(p=p_*(u,v,w)\), and \(\Phi\) is well defined.

Let \((u,v,w)\in\Omega\) and put \(p=p_*(u,v,w)\). Solving
\(h(u,v,w,p)=0\) for \(w\) gives
\[
w=
\left(
\frac{(1-u^p)(1-v^p)}
{\bigl(1-(uv)^p\bigr)^2}
\right)^{1/p},
\]
hence
\[
\Phi\bigl(u,v,p_*(u,v,w)\bigr)=(u,v,w).
\]
Thus \(\Phi\) is surjective. If
\[
\Phi(u_1,v_1,p_1)=\Phi(u_2,v_2,p_2),
\]
then \(u_1=u_2\) and \(v_1=v_2\). Writing \(w\) for the common third
component, the first part of the proof gives
\[
p_1=p_*(u_1,v_1,w)
   =p_*(u_2,v_2,w)
   =p_2.
\]
Hence \(\Phi\) is injective, and
\[
\Phi^{-1}(u,v,w)
=
\bigl(u,v,p_*(u,v,w)\bigr).
\]
The map \(\Phi\) is of class \(C^\infty\), and by
Theorem~\ref{thm:smooth-threshold}, so is its inverse. Hence
\(\Phi\) is a \(C^\infty\)-diffeomorphism.
\end{proof}
%%%%%%%%%%%%%%%%%%%%%%%%%%%%%%%%%%%%%%%%%%%%%%%%%%%%%%%%
As an application of the preceding results, consider the one-parameter
family
\[
f_{\varepsilon}(x)
=
(x+a)^3(x^2+2\varepsilon x+1),
\qquad
a>0,\;0<\varepsilon<1.
\]
Its zeros are
\(-a,-a,-a,-\varepsilon\pm i\sqrt{1-\varepsilon^2}\), so the
complex-conjugate pair approaches the imaginary axis as
\(\varepsilon\to0^+\), while
\(f_{\varepsilon}\in H_5\) for every
\(\varepsilon\in(0,1)\). Expanding \(f_{\varepsilon}\) gives
\[
f_{\varepsilon}(x)
=
x^5+(3a+2\varepsilon)x^4
+(3a^2+6a\varepsilon+1)x^3
+(a^3+6a^2\varepsilon+3a)x^2
+(2a^3\varepsilon+3a^2)x+a^3,
\]
and hence determines the smooth curve
\[
\Gamma_a(\varepsilon)
=
(u_\varepsilon,v_\varepsilon,w_\varepsilon),
\qquad
0<\varepsilon<1,
\]
where
\[
u_{\varepsilon}
=
\frac{a^3+6a^2\varepsilon+3a}
{(3a+2\varepsilon)(3a^2+6a\varepsilon+1)},
\qquad
v_{\varepsilon}
=
\frac{(3a^2+6a\varepsilon+1)a^3}
{(a^3+6a^2\varepsilon+3a)(2a^3\varepsilon+3a^2)},
\]
and
\[
w_{\varepsilon}
=
\frac{(3a+2\varepsilon)(2a^3\varepsilon+3a^2)}
{(3a^2+6a\varepsilon+1)(a^3+6a^2\varepsilon+3a)}.
\]

Since
\(\Gamma_a:(0,1)\to\Omega\)
is smooth, Theorem~\ref{thm:smooth-threshold} shows that the function
\[
\varepsilon
\longmapsto
p_*\bigl(\Gamma_a(\varepsilon)\bigr)
=
p_*(f_\varepsilon)
\]
is of class \(C^\infty\).
Moreover,
\(p_*(f_\varepsilon)\)
is characterized implicitly by
\[
H_a\!\left(p_*(f_\varepsilon),\varepsilon\right)=0,
\]
where
\[
H_a(p,\varepsilon)
=
\ln(1-u_\varepsilon^p)
+\ln(1-v_\varepsilon^p)
-2\ln\!\bigl(1-(u_\varepsilon v_\varepsilon)^p\bigr)
-p\ln w_\varepsilon.
\]

Differentiation with respect to \(\varepsilon\) gives
\[
\frac{d}{d\varepsilon}p_*(f_\varepsilon)
=
-
\frac{
\partial_\varepsilon H_a
\!\left(p_*(f_\varepsilon),\varepsilon\right)}
{
\partial_p H_a
\!\left(p_*(f_\varepsilon),\varepsilon\right)}.
\]

Figure~\ref{fig:threshold-boundary} illustrates the variation of the
stability threshold along the trajectory
\(\Gamma_a(\varepsilon)\) for representative values of \(a\).

\begin{figure}[htbp]
\centering
\includegraphics[width=0.82\textwidth]{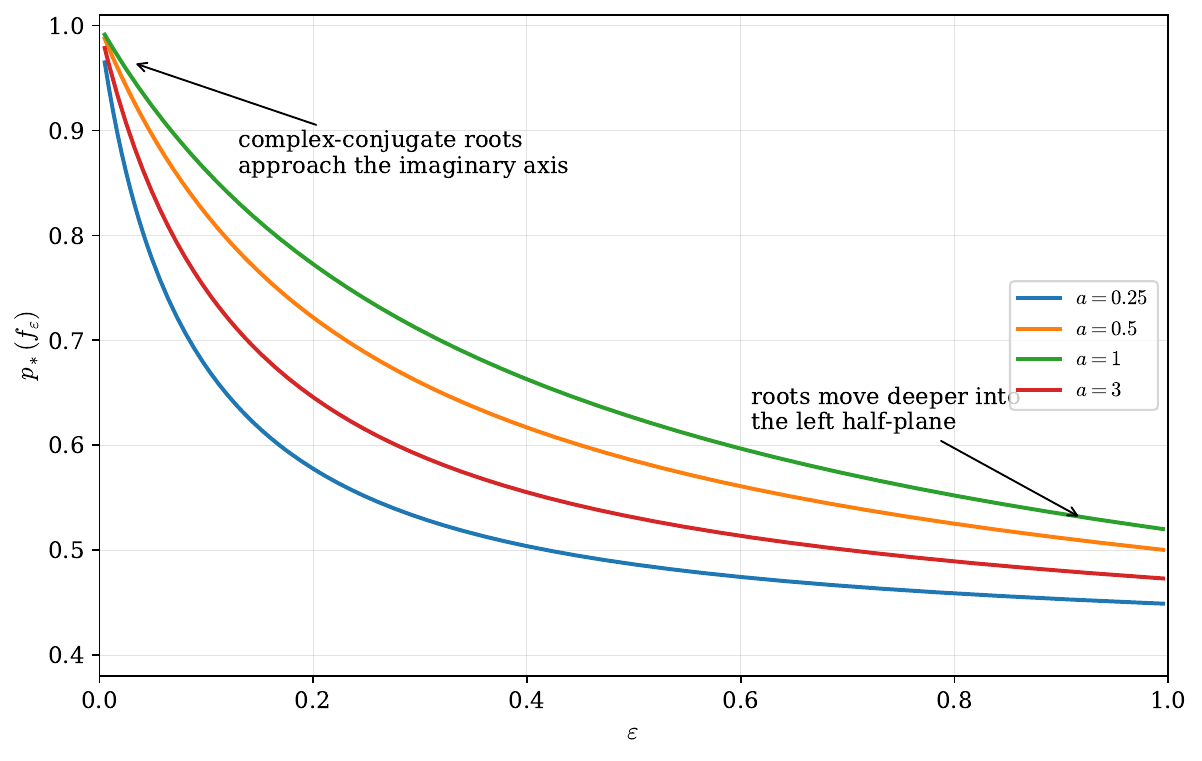}
\caption{Variation of the stability threshold along the
trajectory \(\Gamma_a(\varepsilon)\) associated with
\(f_\varepsilon(x)=(x+a)^3(x^2+2\varepsilon x+1)\) for representative
values of \(a\).}
\label{fig:threshold-boundary}
\end{figure}
\section{Conclusion} \label{sec:conclusion}

This paper provides a complete characterization of the
stability-preserving exponent set for fractional Hadamard powers of
degree-five stable polynomials with positive coefficients.
Using a  parameterization of the Hurwitz region, the
stability problem is reduced to the analysis of a single logarithmic
scalar function. This reduction leads to the existence of a unique
stability threshold \(p_*(f)\in(0,1)\), and consequently to the complete
description $
\mathcal P(f)=\bigl(p_*(f),\infty\bigr)$
for every polynomial \(f\in H_5\).

The  formulation also reveals additional geometric structure.
The stability threshold depends smoothly on the  parameters,
and the threshold itself provides a global smooth coordinate on the
admissible parameter region. The one-parameter family considered in
Section~\ref{sec:smooth-threshold} illustrates this geometric
interpretation by showing how the threshold evolves continuously as the 
pair of complex-conjugate zeros approaches the imaginary axis.

The characterization obtained in this paper relies on a feature that is
specific to degree five, namely that the
stability problem is governed by a single Hurwitz determinant.
Beginning with degree six, several Hurwitz determinants interact
simultaneously, making an analogous reduction substantially more
difficult. Nevertheless, the smooth geometric viewpoint developed in this paper provides a starting point for extending these ideas to higher degrees.
\section*{Data Availability Statement}

The data on which Figure~1 is based are available on request from the first author.

\section*{Declaration of Funding}
The work of the third author was supported by the SRP Program of the HTWG Konstanz - University of Applied Sciences.

\bibliographystyle{plain}
\bibliography{references}
\end{document}